\documentclass[11pt]{amsart}
\usepackage[top=1in, bottom=1in, left=1in, right=1in]{geometry}
\usepackage[T1]{fontenc} 
\usepackage[utf8]{inputenc} 
\usepackage{amsmath, amsthm, amsfonts}
\usepackage{enumerate}
\usepackage{hyperref}
\usepackage{amssymb}
\usepackage{lipsum}
\usepackage{xcolor}
\usepackage{fancyhdr, textcomp}
\usepackage{tikz, tikz-cd, quiver}
\usepackage{graphicx}
\usepackage[all]{xy}

\newtheorem{thm}{Theorem}[]

\newtheorem{lem}[thm]{Lemma}

\theoremstyle{definition}

\theoremstyle{remark}

\newcommand{\HS}{\text{HS}}

\makeatletter
\let\origsetaddresses\@setaddresses

\def\@setaddresses{}
\makeatother

\allowdisplaybreaks
\title{An algebraic proof of the Milnor-Orlik theorem}

\author{Yerly Soler}
\address{CENTRO DE CIENCIAS MATEMÁTICAS, UNAM, Morelia, México}
\email{yerly@matmor.unam.mx}

\date{}
\begin{document}

\thanks{
\mbox{\hspace{11pt}}{\it Key words}: Milnor number, weighted homogeneous polynomials, Hilbert series, Koszul complex.\\
\mbox{\hspace{11pt}}Research supported by SECIHTI project CF-2023-G33.}

\begin{abstract}
A well-known theorem by Milnor-Orlik provides a formula for the Milnor number of a weighted-homogeneous polynomial having an isolated singularity that depends only on the weights. In this paper we present a proof of that result  using techniques from commutative algebra. Our approach is to obtain a free resolution of the Milnor algebra through the Koszul complex. The desired formula is then obtained from a Hilbert series calculation.
\end{abstract}

\maketitle

\section{Introduction}
The Milnor number is a fundamental analytic invariant associated to a complex hypersurface isolated singularity. When the defining equation of the hypersurface is a weighted-homogeneous polynomial, Milnor and Orlik provided an explicit formula for computing the Milnor number that depends only on the weight and not on the polynomial itself \cite{MILNOR1970}. This result, established in 1970, was proved using techniques from algebraic topology and relied on the topological structure of the corresponding Milnor fiber. Since it is known that the Milnor fiber has the homotopy type of a bouquet of spheres, their approach consists on the computation of the rank of a certain homology group.

Besides the original proof, in the literature there are at least two other approaches to the Milnor-Orlik Theorem. One of them is the geometric approach of Eyral and Oka \cite{Eyral_Oka2024}, which establishes a weak version of that theorem. They show that, under certain hypothesis, the monodromy zeta-function (and hence the Milnor number) is completely determined by the weight. The other is a combinatorial proof by Boubakri, Greuel and Markwig \cite{GreuelMarkwig2012}. In their work, a positive-characteristic version of the Milnor-Orlik Theorem is given using the Newton polyhedron as the main tool. 

In this paper we present an algebraic proof of the Milnor-Orlik Theorem. In summary, the strategy is the following. For a polynomial defining a hypersurface with an isolated singularity, we use the fact that its first-order partial derivatives form a regular sequence. Consequently, the Koszul complex gives a free resolution of the Milnor algebra. Using the additivity of Hilbert series, a few elementary computations lead to the desired formula.
\section{Preliminaries}
Let $f:(\mathbb{C}^r,0)\rightarrow(\mathbb{C},0)$ be a holomorphic function with an isolated singularity at the origin. The \textit{Milnor algebra}  of $f$ is defined as 
$$M_f=\dfrac{\mathbb{C}\{x_1,\ldots,x_r\}}{\langle f_1,\ldots,f_r\rangle},$$ 
where $f_i:=\partial f/\partial x_i$ for $i=1,\ldots,r.$ The \textit{Milnor number of} $f$ is given by $\mu(f)=\dim_{\mathbb{C}}M_f.$

\medskip

A polynomial $f\in\mathbb{C}[x_1,\ldots,x_r]$ is called \textit{weighted homogeneous} if there exists a vector $w=(w_1,\ldots,w_r)\in\mathbb{Q}^r_{>0}$ and $d\in\mathbb{N}$ such that every monomial $x^{\alpha}$ of $f$ with nonzero coefficient satisfies $\sum_{i=1}^{r}\alpha_iw_i=d$. In this case we say that $f$ is of type $(w_1,\ldots,w_r;d).$ We call  $w_1,\ldots,w_r$ the weights of the variables and $d$ the weighted degree of $f,$ denoted as $\deg_w(f)=d$. 

Notice that every weighted-homogeneous polynomial of type $(w_1,\ldots,w_r;d)$ can also be of type $(w'_1,\ldots,w'_r;d')$ where each $w'_i\in\mathbb{N}.$ From now on we assume that all weights are positive integers. Moreover, notice that the partial derivatives $f_i$ of a weighted-homogeneous polynomial $f(x_1,\ldots,x_r)$ of type $(w_1,\ldots,w_r;d)$ is itself weighted homogeneous of type $(w_1,\ldots,w_r; d-w_i).$\medskip

Fix a weight $w=(w_1,\ldots,w_r)$. We endow the polynomial ring $S=\mathbb{C}[x_1,\ldots,x_r]$ with the $w$-\textit{grading} (or \textit{weighted grading}) determined by assigning $\deg_w(x_i)=w_i$ for each $i=1,\ldots,r.$ For a monomial $x^{m}=x_1^{m_1}\cdots x_r^{m_r},$ its weighted degree is $\deg_w(x^m)=\sum_{i=1}^r m_iw_i.$ With this grading, the ring decomposes as 
$$S=\displaystyle\bigoplus_{\alpha\in\mathbb{N}}S_{\alpha},$$
where $S_{\alpha}$ denotes the $\mathbb{C}$-vector space of all weighted homogeneous polynomials of $w$-degree $\alpha.$  \\

We can now state the Milnor-Orlik Theorem.

\begin{thm}[Milnor-Orlik \cite{MILNOR1970}]\label{thm:M-O}
    Let $f(x_1,\ldots,x_r)$ be a weighted homogeneous polynomial of type $(w_1,w_2,\ldots,w_r;d)$ defining an isolated singularity at the origin. Then the Milnor number of $f$ can be computed as follows:
    $$\mu(f)=\displaystyle\prod_{i=1}^{r}\dfrac{d-w_i}{w_i}.$$
\end{thm}

\section{An algebraic proof of Theorem \ref{thm:M-O}}

\begin{lem}\label{lem:Suma}
    Let $r, d\in\mathbb{N}$, and $w_1, w_2,\ldots,w_r\in\mathbb{N}$ such that $d\geq w_i$ for each $i$. Then the following equality holds:
    $$\prod_{i=1}^{r}(1-t^{d-w_i})=1+\displaystyle\sum_{k=1}^{r}(-1)^k\sum_{\substack{i_1<\cdots<i_k\\1\leq i_j\leq r}}t^{kd-(w_{i_1}+\cdots+w_{i_k})}.$$
\end{lem}
\begin{proof}
We proceed by induction on \( r \). For $r = 1$, the statement is clear. Assume that the statement is true for $r,$ we shall prove it for $r+1.$ Indeed, 

\begin{align*}
       \prod_{i=1}^{r+1}(1-t^{d-w_i})=&(1-t^{d-w_{r+1}})\prod_{i=1}^{r}(1-t^{d-w_i})\\
       	{=}&(1-t^{d-w_{r+1}})\left(1+\displaystyle\sum_{k=1}^{r}(-1)^k\sum_{\substack{i_1<\cdots<i_k\\1\leq i_j\leq r}}t^{kd-(w_{i_1}+\cdots+w_{i_k})}\right).
     \end{align*}
 Expanding the product, the last expression becomes:
 $$1+\left(\displaystyle\sum_{k=1}^{r}(-1)^k\sum_{\substack{i_1<\cdots<i_k\\1\leq i_j\leq r}}t^{kd-(w_{i_1}+\cdots+w_{i_k})}\right)-t^{d-w_{r+1}}
     +\displaystyle\sum_{k=1}^{r}(-1)^{k+1}\sum_{\substack{i_1<\cdots<i_k\\1\leq i_j\leq r}}t^{d-w_{r+1}+kd-(w_{i_1}+\cdots+w_{i_k})}.$$
     
We rearrange the previous sum by putting together the terms involving $t^{d-w_i}$, $i\in\{1,\ldots,r+1\}$, and isolating the summand corresponding to $k=r$ from the second sum. Then the previous sum turns into:

\begin{align*}
      &1-\left(\displaystyle\sum_{i=1}^{r+1}t^{d-w_{i}}\right)+\left(\displaystyle\sum_{k=2}^{r}(-1)^k\sum_{\substack{i_1<\cdots<i_k\\1\leq i_j\leq r}}t^{kd-(w_{i_1}+\cdots+w_{i_k})}\right)+\\
     &+\displaystyle\sum_{k=1}^{r-1}(-1)^{k+1}\sum_{\substack{i_1<\cdots<i_k\\1\leq i_j\leq r}}t^{(k+1)d-(w_{i_1}+\cdots+w_{i_k}+w_{r+1})}+(-1)^{r+1}t^{(r+1)d-(w_{1}+\cdots+w_{r+1})}.\\
     \end{align*}
     Reindexing the fourth term gives:
     \begin{align*}
     &1-\left(\displaystyle\sum_{i=1}^{r+1}t^{d-w_{i}}\right)+\left(\displaystyle\sum_{k=2}^{r}(-1)^k\sum_{\substack{i_1<\cdots<i_k\\1\leq i_j\leq r}}t^{kd-(w_{i_1}+\cdots+w_{i_k})}\right)+\\
     &+\displaystyle\sum_{k=2}^{r}(-1)^{k}\sum_{\substack{i_1<\cdots<i_{k-1}\\1\leq i_j\leq r}}t^{kd-(w_{i_1}+\cdots+w_{i_{k-1}}+w_{r+1})}+(-1)^{r+1}t^{(r+1)d-(w_1+\cdots+w_{r+1})}=(\ast).
   \end{align*}
   On the other hand, returning to the expression we are trying to prove, notice that the $k$th term of $\displaystyle\sum_{k=1}^{r+1}(-1)^k\sum_{\substack{i_1<\cdots<i_k\\1\leq i_j\leq r+1}}t^{kd-(w_{i_1}+\cdots+w_{i_k})}$ with $2\leq k\leq r$ can also be written as $$(-1)^k\displaystyle\sum_{\substack{i_1<\cdots<i_k\\1\leq i_j\leq r}}t^{kd-(w_{i_1}+\cdots+w_{i_k})}+(-1)^k\displaystyle\sum_{\substack{i_1<\cdots<i_{k-1}\\1\leq i_j\leq r}}t^{kd-(w_{i_1}+\cdots+w_{i_{k-1}}+w_{r+1})}.$$
   Hence, the expression $(\ast)$ is equal to 
   \begin{align*}
     &=1-\left(\displaystyle\sum_{i=1}^{r+1}t^{d-w_{i}}\right)+\displaystyle\sum_{k=2}^{r}(-1)^k\sum_{\substack{i_1<\cdots<i_k\\1\leq i_j\leq r+1}}t^{kd-(w_{i_1}+\cdots+w_{i_k})}+(-1)^{r+1}t^{(r+1)d-(w_1+\cdots+w_{r+1})}\\
     &=1+\displaystyle\sum_{k=1}^{r+1}(-1)^k\sum_{\substack{i_1<\cdots<i_k\\1\leq i_j\leq r+1}}t^{kd-(w_{i_1}+\cdots+w_{i_k})}.
   \end{align*}
\end{proof}

Let us briefly recall some basic facts about shifted graded modules that we need for the proof of Theorem \ref{thm:M-O}.

Let $M=\bigoplus_{\alpha\in\mathbb{N}}M_\alpha$ be a graded $S$-module. For a positive integer $a,$ the \textit{shifted module} $M(-a)$ is defined by $M(-a)=\bigoplus_{\alpha\in\mathbb{N}}M(-a)_\alpha$ where each component is given by $M(-a)_\alpha:=M_{\alpha-a},$ and $M(-a)_j:=0$ for all $j<a.$ \medskip

\begin{proof}[Proof of Theorem \ref{thm:M-O}]

Let $S=\mathbb{C}[x_1,\ldots,x_r]$, considered as a  $w$-graded ring, where $w=(w_1,\ldots,w_r)$. Recall that the  partial derivatives $f_i$ of $f$ are homogeneous elements of $S,$ $\deg_w(f_i)=d-w_i$. Consider the Koszul complex 
$K_{\cdot}(\mathbf{x})$, where $\mathbf{x}=(f_1,\ldots,f_r)$:
\begin{center}
 \small{
 \begin{tikzcd}
 0 \arrow{r} & S(-(rd - w_1 - \ldots - w_r)) \arrow{r} & \displaystyle \bigoplus_{\substack{i_1 < \cdots< i_{r-1} \\ 1 \leq i_j \leq r}} S(-(kd - w_{i_1} - \ldots - w_{i_{r-1}})) \arrow{r} &\cdots
 \end{tikzcd}	
\begin{tikzcd}	
\cdots \arrow{r} & \displaystyle \bigoplus_{\substack{i_1 < i_2 \\ 1 \leq i_j \leq r}} S(-(2d - w_{i_1} - w_{i_2})) \arrow{r} & \displaystyle \bigoplus_{1 \leq i_1 \leq r} S(-(d - w_{i_1})) \arrow{r} & S \arrow{r} & M_f \arrow{r} & 0.
\end{tikzcd}
}
\end{center}
Since $f$ defines an isolated hypersurface singularity, $f_1,\ldots,f_r$ is a regular sequence. Therefore, the Koszul complex is an exact sequence of graded free $S$-modules.
     \medskip

Recall that $M_f$ inherits a natural $w$-grading from the graded ring $S=\displaystyle\bigoplus_{\alpha\in\mathbb{N}}S_\alpha$. Now consider the Hilbert series of $M_f$,
    $$\HS_{M_f}(t):=\sum_{\alpha=0}^\infty\dim_{\mathbb{C}}(M_f)_\alpha\ t^\alpha.$$
Applying the additivity of the Hilbert series to the previous exact sequence we obtain:
    $$\HS_{M_f}(t)=\HS_S(t)+\displaystyle\sum_{k=1}^{r}(-1)^k\sum_{\substack{i_1<\cdots<i_k\\1\leq i_j\leq r}}\HS_{S(-(kd-w_{i_1}-\ldots-w_{i_k}))}(t).$$
Since $f$ has an isolated singularity at the origin, $\dim_{\mathbb{C}}M_f$ is finite. Hence $\HS_{M_f}(t)$ has a finite number of non-zero summands. It follows that $\HS_{M_f}(1)=\dim_{\mathbb{C}}M_f=\mu(f).$ 
    \medskip
    
Let us denote $$D(\alpha):=\dim_{\mathbb{C}} S_\alpha=\lvert\{(\alpha_1,\alpha_2,\ldots,\alpha_r)\in\mathbb{Z}^{r}_{\geq0}:\ \displaystyle\sum_{i=1}^{r}w_i\alpha_i=\alpha\}\rvert.$$ 

Hence, $\HS_S(t)=\displaystyle\sum_{\alpha=0}^{\infty}D(\alpha)t^\alpha.$ From the definition of $D(\alpha),$ it follows that 
$$\HS_S(t)=\left(\displaystyle\sum_{\alpha_1=0}^{\infty}t^{w_1\alpha_1}\right)\left(\displaystyle\sum_{\alpha_2=0}^{\infty}t^{w_2\alpha_2}\right)\cdots\left(\displaystyle\sum_{\alpha_r=0}^{\infty}t^{w_r\alpha_r}\right)=\prod_{i=1}^{r}\dfrac{1}{1-t^{w_i}}.$$ 
Consequently, 
    $$\left(\prod_{i=1}^{r}(1-t^{d-w_i})\right)\HS_S(t)=\prod_{i=1}^{r}\dfrac{1-t^{d-w_i}}{1-t^{w_i}}.$$ 
According to Lemma \ref{lem:Suma}, the left-hand side of the previous expression is equal to 
    \begin{align*}
      \left(\prod_{i=1}^{r}(1-t^{d-w_i})\right)\HS_S(t)&=\left(1+\displaystyle\sum_{k=1}^{r}(-1)^k\sum_{\substack{i_1<\cdots<i_k\\1\leq i_j\leq r}}t^{kd-(w_{i_1}+\cdots+w_{i_k})}\right)\displaystyle\sum_{\alpha=0}^{\infty}D(\alpha)t^\alpha\\
      &=\displaystyle\sum_{\alpha=0}^{\infty}D(\alpha)t^\alpha+\displaystyle\sum_{k=1}^{r}(-1)^k\sum_{\substack{i_1<\cdots<i_k\\1\leq i_j\leq r}}\left(\displaystyle\sum_{\alpha=0}^{\infty}D(\alpha)t^{\alpha+kd-(w_{i_1}+\cdots+w_{i_k})}\right).
       \end{align*}
Adjusting the index of the second infinite sum, and using that $D(\alpha-a)=\dim_{\mathbb{C}}S(-a)_\alpha$, the last expression becomes:
       \begin{align*}
       &\displaystyle\sum_{\alpha=0}^{\infty}D(\alpha)t^\alpha+\displaystyle\sum_{k=1}^{r}(-1)^k\sum_{\substack{i_1<\cdots<i_k\\1\leq i_j\leq r}}\left(\displaystyle\sum_{\alpha=kd-w_{i_1}-\ldots-w_{i_k}}^{\infty}D(\alpha-(kd-w_{i_1}-\ldots-w_{i_k}))t^{\alpha}\right)\\
           &=\HS_S(t)+\displaystyle\sum_{k=1}^{r}(-1)^k\sum_{\substack{i_1<\cdots<i_k\\1\leq i_j\leq r}}\HS_{S(-(kd-w_{i_1}-\ldots-w_{i_k}))}(t)=\HS_{M_f}(t).
       \end{align*}
Therefore, $\HS_{M_f}(t)=\left(\displaystyle\prod_{i=1}^{r}(1-t^{d-w_i})\right)\HS_S(t)=\displaystyle\prod_{i=1}^{r}\dfrac{1-t^{d-w_i}}{1-t^{w_i}}.$\medskip
       
As $M_f$ is finite-dimensional, the Hilbert series $\HS_{M_f}(t)$ is a polynomial. Taking limits on the last expression we obtain,
    $$\mu(f)=\HS_{M_f}(1)=\displaystyle\lim_{t\rightarrow 1}\HS_{M_f}(t)=\displaystyle\lim_{t\rightarrow 1}\prod_{i=1}^{r}\dfrac{1-t^{d-w_i}}{1-t^{w_i}}.$$
      Since the limit of each term in the product exists, $\displaystyle\lim_{t\rightarrow1}\dfrac{1-t^{d-w_i}}{1-t^{w_i}}=\dfrac{d-w_i}{w_i}.$ 
       Hence  
  $$\mu(f)=\HS_{M_f}(1)=\displaystyle\lim_{t\rightarrow 1}\prod_{i=1}^{r}\dfrac{1-t^{d-w_i}}{1-t^{w_i}}=\prod_{i=1}^{r}\displaystyle\lim_{t\rightarrow 1}\dfrac{1-t^{d-w_i}}{1-t^{w_i}}=\prod_{i=1}^{r}\dfrac{d-w_i}{w_i}.$$
\end{proof}

\section*{Acknowledgements}

We thank Guillermo Peñafort for suggesting the idea of using the Koszul complex in the study of the Milnor-Orlik Theorem. We are also grateful to Hernán de Alba for discussing an early version of this manuscript and for his valuable suggestions.
\bibliographystyle{alpha}
\bibliography{references}

\makeatletter
\origsetaddresses   
\makeatother

\end{document}